\documentclass[10pt,reqno]{amsart}
     \makeatletter
     \def\section{\@startsection{section}{1}%
     \z@{.7\linespacing\@plus\linespacing}{.5\linespacing}%
     {\bfseries
     \centering
     }}
     \def\@secnumfont{\bfseries}
     \makeatother
\newtheorem{theorem}{Theorem}[section]
\newtheorem{lemma}[theorem]{Lemma}
\newtheorem{proposition}[theorem]{Proposition}

\theoremstyle{definition}
\newtheorem{definition}[theorem]{Definition}
\newtheorem{example}[theorem]{Example}
\theoremstyle{remark}
\newtheorem{remark}[theorem]{Remark}
\numberwithin{equation}{section}
\usepackage{latexsym,amssymb}

\usepackage[pdftitle={E-semigroups Subordinate to CCR Flows},
            pdfauthor={Stephen J. Wills},
            bookmarks=true,
            bookmarksnumbered=true,
            colorlinks=false,
            dvips=true,
            pdfstartview=FitH]{hyperref}

\newcommand{\Hil}{\mathsf{H}}

\newcommand{\kil}{\mathsf{k}}

\newcommand{\Fock}{\mathcal{F}}
\newcommand{\Exp}{\mathcal{E}}
\newcommand{\e}[1]{\varepsilon (#1)}

\newcommand{\init}{\mathfrak{h}}
\newcommand{\noise}{\mathsf{k}}
\newcommand{\khat}{{\wh{\noise}}}

\newcommand{\ind}{\mathbf{1}}

\newcommand{\Alg}{\mathcal{A}}
\newcommand{\Blg}{\mathcal{B}}

\newcommand{\vM}{\mathcal{M}}
\newcommand{\vN}{\mathcal{N}}

\newcommand{\al}{\alpha}
\newcommand{\be}{\beta}
\newcommand{\ga}{\gamma}

\newcommand{\de}{\delta}
\newcommand{\si}{\sigma}

\newcommand{\Real}{\mathbb{R}}
\newcommand{\Rplus}{\Real_+}
\newcommand{\Comp}{\mathbb{C}}

\newcommand{\ip}[2]{\langle #1, #2 \rangle}

\newcommand{\norm}[1]{\lVert #1 \rVert}
\newenvironment{sbmatrix}
{\bigl[\begin{smallmatrix}}{\end{smallmatrix}\bigr]}
\newenvironment{spmatrix}
{\bigl(\begin{smallmatrix}}{\end{smallmatrix}\bigr)}
\newcommand{\bra}[1]{\langle #1 |}
\newcommand{\ket}[1]{| #1 \rangle}

\newcommand{\wh}{\widehat}

\newcommand{\ot}{\otimes}
\newcommand{\op}{\oplus}
\newcommand{\uot}{\underline{\otimes}}
\newcommand{\oot}{\overline{\otimes}}
\newcommand{\vot}{\oot}
\newcommand{\aot}{\uot}

\newcommand{\To}{\rightarrow}

\newcommand{\ges}{\geqslant}
\newcommand{\les}{\leqslant}
\newcommand{\cles}{\precsim}

\newcommand{\tu}{\textup}

\DeclareMathOperator{\Ran}{Ran}

\DeclareMathOperator{\Ker}{Ker}
\DeclareMathOperator{\id}{id}

\DeclareMathOperator{\Sub}{Sub}

\newenvironment{alist}
{

\begin{enumerate}}
{\end{enumerate}}

\newenvironment{rlist}
{

\begin{enumerate}}
{\end{enumerate}}

\newcounter{step_count}

\numberwithin{equation}{section}

\begin{document}

\title[$E$-semigroups Subordinate to CCR Flows]{$E$-semigroups
Subordinate to CCR Flows}

\author{Stephen J.~Wills}
\address{School of Mathematical Sciences, University College Cork,
Cork, Ireland}
\email{s.wills@ucc.ie}
\dedicatory{Dedicated to Robin Hudson on his 70th birthday --- an
inspirational academic grandfather.}



\begin{abstract}
The subordinate $E$-semigroups of a fixed $E$-semigroup $\al$ are in
one-to-one correspondence with local projection-valued cocycles of
$\al$. For the CCR flow we characterise these cocycles in terms of
their stochastic generators, that is, in terms of the coefficient
driving the quantum stochastic differential equation of
Hudson-Parthasarathy type that such cocycles necessarily satisfy. In
addition various equivalence relations and order-type relations on
$E$-semigroups are considered, and shown to work especially well in
the case of those semigroups subordinate to the CCR flows by
exploiting our characterisation.
\end{abstract}

\maketitle

\section{Introduction}

Two strands of noncommutative analysis developed contemporaneously in
the 1980's: within the field of quantum probability there was great
interest in quantum stochastic calculi, and especially successful was
the bosonic calculus of Hudson and Parthasarathy (\cite{HuP}).
Meanwhile, Arveson and Powers considered the problem of understanding
\emph{endomorphism} or \emph{$E$-semigroups} on $B(\Hil)$, the algebra
of all bounded operators on a Hilbert space $\Hil$, thereby
generalising Wigner's well-known work for automorphism groups from the
1930's (\cite{Ebook}). Significant relations between the two were
explored by Bhat some time later (\cite{ccr}), exploiting the fact
that a central object for the quantum probabilists is an archetypal
example of an $E$-semigroup, namely the right shift semigroup on
(symmetric) Fock space, otherwise known as the \emph{CCR flow}. Bhat
in particular initiated the study of various order structures, looking
at \emph{local positive contraction cocycles} for the CCR flow since
these characterise the completely positive (CP) semigroups subordinate
to the given $E$-semigroup. In this paper we shall continue our study
of cocycles of the CCR flow from~\cite{jtl}, in particular studying
the order relation on orthogonal projection-valued cocycles, and
determining which of these cocycles are equivalent in a natural sense.

To make matters more concrete, we give some definitions.

\begin{definition}
Let $\Hil$ be a Hilbert space. An \emph{$E$-semigroup} on $B(\Hil)$ is
a family of maps $\al = (\al_t:B(\Hil) \To B(\Hil))_{t \ges 0}$ that
satisfy:
\begin{rlist}
\item
$\al_0 = \id_{B(\Hil)}$ and $\al_{s+t} = \al_s \circ \al_t$ for all
$s,t \ges 0$;
\item
$\al_t$ is an endomorphism of $B(\Hil)$, i.e. a ${}^*$-homomorphism,
for each $t \ges 0$;
\item
the map $S \mapsto \al_t(S)$ is normal for each $t \ges 0$, and the
map $t \mapsto \al_t (S)$ is continuous in the weak operator topology
for each $S \in B(\Hil)$.
\end{rlist}
If, in addition, we have
\begin{rlist}
\item[(iv)]
$\al_t (I) = I$ for all $t$, where $I$ is the identity operator on
$\Hil$,
\end{rlist}
then $\al$ is called an \emph{$E_0$-semigroup}.
\end{definition}

The theory of these has proved to be significantly more tricky than
the subclass of \emph{automorphism} semigroups; dropping the
requirement of Wigner that $\Ran \al_t = B(\Hil)$ is a nontrivial
business. Recent work of Powers et al.\ (\cite{poNY, app, 2-1}) has
established a new route to the construction of \emph{spatial}
$E_0$-semigroups through the study of CP flows, and as with Bhat's
work, local cocycles and subordination are a key ingredient.

\begin{definition}
Let $\Hil$ be a Hilbert space and let $\al, \be:B(\Hil) \To B(\Hil)$
be a pair of completely positive maps. We say that $\al$
\emph{dominates} $\be$, or, equivalently, that $\be$ \emph{is
subordinate to} $\al$, if their difference $\al-\be$ is a completely
positive map, and denote this $\al \ges \be$, or $\be \les \al$.

Similarly, if instead $\al = (\al_t)_{t \ges 0}$ and $\be = (\be_t)_{t
\ges 0}$ are $E$-semigroups, then $\be$ is a \emph{subordinate} of
$\al$, written $\be \les \al$, if $\al_t - \be_t$ is completely
positive for all $t \ges 0$.
\end{definition}

The notion of domination or subordination goes back at least as far as
work of Arveson. More recent generalisations include subordination in
the theory of completely positive definite kernels (i.e.\ maps $k:S
\times S \To B(\Alg;\Blg)$ satisfying a suitable positivity
requirement, where $S$ is a set and $\Alg$, $\Blg$ are $C^*$-algebras)
and semigroups of these (\cite{beast}, Section~3.3). A
characterisation of the subordinates of a given CP map $\al$ is the
content of Arveson's Radon-Nikod\'{y}m Theorem (\cite{subalg},
Theorem~1.4.2). In this paper we will look only at subordination of
one $E$-semigroup by another, which gives the opportunity to provide
the following short proof of Arveson's result in this setting.

\begin{lemma}\label{subord maps}
Let $\Hil$ be a Hilbert space and let $\al, \be$ be a pair of
endomorphisms on $B(\Hil)$. The following are equivalent\tu{:}
\begin{rlist}
\item
$\al \ges \be$\tu{;}
\item
$\be (I) \in \al \bigl(B(\Hil)\bigr)'$ and $\be (S) = \be(I) \al(S)$
for all $S \in B(\Hil)$.
\end{rlist}
Furthermore, given any projection $P \in \al \bigl(B(\Hil)\bigr)'$ the
map $\ga: S \mapsto P\al(S)$ is a subordinate of $\al$.
\end{lemma}

\begin{remark}
Throughout, unless otherwise specified, or obvious, all projections
on Hilbert spaces will be \emph{orthogonal} projections.
\end{remark}

\begin{proof}
Suppose that $\al \ges \be$. Let $S \in B(\Hil)$ then
\begin{align*}
0 & \les \begin{bmatrix} I & S^* \end{bmatrix} \begin{bmatrix} I \\ S
\end{bmatrix} = \begin{bmatrix} I & S \\ S^* & S^*S \end{bmatrix} \\
\intertext{and since the difference $\al-\be$ is
$2$-positive we have}
0 & \les \begin{bmatrix} \be(I) & 0 \\ 0 & I \end{bmatrix}^*
(\al-\be)^{(2)} \biggl( \begin{bmatrix} I & S \\ S^* & S^*S
\end{bmatrix} \biggr) \begin{bmatrix} \be(I) & 0 \\ 0 & I \end{bmatrix}
\\
&= \begin{bmatrix} \be(I)\al(I)\be(I) -\be(I)^3 & \be(I) \bigl(
\al(S)-\be(S) \bigr) \\ \bigl( \al(S^*)-\be(S^*) \bigr) \be(I) &
\al(S^*S) -\be(S^*S) \end{bmatrix} =: R.
\end{align*}
But $\al(I) \ges \be(I)$, and these operators are projections, hence
\[
\be(I)\al(I)\be(I) = \be(I) = \be(I)^3.
\]
Thus the top-left hand corner of the positive $2 \times 2$ operator
matrix $R$ is zero. Writing this as $R = T^*T$ for $T = \left[
\begin{smallmatrix} A & B \\ C & D \end{smallmatrix} \right]$, where
$A,B,C,D \in B(H)$, it follows that $A = C = 0$, and so the
off-diagonal entries of $R$ are also zero. Hence $\be(I) \al(S) =
\be(S) = \al(S) \be(I)$ and thus~(i) implies~(ii).

That (ii) implies (i) follows from the final part of the lemma
concerning the map $\ga$, and this follows since for such projections
$P$ we have
\[
\al(S) - \ga(S) = P^\perp \al(S) P^\perp,
\]
so that $\al-\ga$ is CP.
\end{proof}

\begin{remark}
(a) In proving that (i) implies (ii) we do not use the full force of
the assumption that $\al-\be$ is CP, but in fact only require that
this difference be $2$-positive (and can then conclude that the
difference is in fact CP). Such a feature is also found in Bhat's
generalisation of this lemma (\cite{ccr}, Proposition~4.2) where he
starts with a unital endomorphism $\al$ and a CP map $\be$, and,
assuming only that $\al-\be$ is positive, shows that the difference is
CP, with $\be(S) = \be(I) \al(S)$ where $\be(I) \in \al \bigl( B(H)
\bigr)'$ is a positive contraction.

(b) In the above, if $\al$ and $\be$ are endomorphisms with $\al \ges
\be$, then $\be$ is determined by the value of the projection
$\be(I)$. On the other hand, unless $\al(I) = I$, different
projections in $\al \bigl( B(\Hil) \bigr)'$ can yield the same
subordinate of $\al$. Indeed, if $\ga_1(S) = P_1 \al(S)$ and $\ga_2(S)
= P_2 \al(S)$ are subordinates of $\al$ defined by projections $P_1,
P_2 \in \al \bigl( B(\Hil) \bigr)'$ then $\ga_1 = \ga_2$ if and only
if there are projections $p \les \al(I)$ and $q_1, q_2 \les
\al(I)^\perp$ such that $P_1 = p+q_1$ and $P_2 = p+q_2$.
\end{remark}

\begin{example}
Let $\Hil$ be a Hilbert space and $L \in B(\Hil)$ an isometry. Let
$\al(S) := LSL^*$, an endomorphism of $B(\Hil)$. Projections $P \in
\al \bigl( B(\Hil) \bigr)'$ are either of the form $LL^*+Q$ or $Q$
where $Q$ is a projection with $Q \les \al(I)^\perp$. The
corresponding subordinates of $\al$ are $\al$ and $0$ respectively.
\end{example}

To apply the lemma to $E$-semigroups in Theorem~\ref{subord Esg} below
(which is a special case of Theorem~4.3 of~\cite{ccr}) we require a
little more terminology.

\begin{definition}
Let $\Hil$ be a Hilbert space and $\al$ an $E$-semigroup on $B(\Hil)$.
A family of contraction operators $X = (X_t)_{t \ges 0}$ on $\Hil$ is
a \emph{left $\al$-cocycle} if
\begin{rlist}
\item
$X_0 = I$;
\item
$X_{s+t} = X_s \al_s (X_t)$ for all $s,t \ges 0$;
\item
the maps $t \mapsto X_t$ and $t \mapsto X^*_t$ are continuous in the
strong operator topology.
\end{rlist}
If, instead, (ii) is replaced by
\begin{rlist}
\item[(ii)$'$] $X_{s+t} = \al_s (X_t) X_s$ for all $s,t \ges 0$
\end{rlist}
then we speak of a \emph{right $\al$-cocycle}. A \emph{local
$\al$-cocycle} is a cocycle for which $X_t \in \al_t
\bigl(B(\Hil)\bigr)'$ for each $t \ges 0$ and so is both a left and
right cocycle.
\end{definition}

\begin{remark}
(a) Contractivity of each $X_t$ is not always assumed; in~\cite{aunt}
we considered cocycles for which the $X_t$ can be \emph{unbounded} (in
which case the continuity condition~(iii) is altered appropriately).
However, the contractivity restriction here is appropriate for this
paper.

(b) Typically it is only strong continuity of $t \mapsto X_t$ that is
assumed, and the class of cocycles is restricted (e.g.\ to
unitary-valued cocycles). The continuity required above is
automatically true for weak operator measurable isometry-valued
cocycles (\cite{Ebook}, Proposition~2.3.1) and weak operator
measurable positive contraction cocycles (\cite{ccr}, Appendix~A) on
separable Hilbert spaces; also for local cocycles of the CCR flow that
are weakly continuous at $t=0$ (\cite{jtl}, Proposition~2.1) on
arbitrary Hilbert spaces.

(c) Although a local cocycle is both a left and right cocycle, the
converse is not true as can easily be seen from Proposition~3.1
of~\cite{jtl}.
\end{remark}

\begin{theorem}\label{subord Esg}
Let $\Hil$ be a Hilbert space and let $\al$ and $\be$ be
$E$-semigroups on $B(\Hil)$. The following are equivalent\tu{:}
\begin{rlist}
\item
$\al \ges \be$\tu{;}
\item
$B := (B_t = \be_t(I))_{t \ges 0}$ is a projection-valued, local
$\al$-cocycle such that $\be_t(S) = B_t \al_t(S)$ for all $S \in
B(\Hil)$.
\end{rlist}
Moreover, if $C = (C_t)_{t \ges 0}$ is any projection-valued, local
$\al$-cocycle, then $\ga_t(S) := C_t \al_t(S)$ defines an
$E$-semigroup subordinate to $\al$.
\end{theorem}

\begin{proof}
Using Lemma~\ref{subord maps}, the proof is now just a matter of
checking that definitions hold. For example if (i) holds then each
$B_t = \be_t(I) \in \al_t \bigl(B(\Hil)\bigr)'$ and is
projection-valued, with
\[
B_{s+t} = \be_{s+t} (I) = \be_s \bigl(\be_t(I)\bigr) = \be_s (I) \al_s
\bigl(\be_t(I)\bigr) = B_s \al_s (B_t).
\]
Since $B_s \les B_t$ whenever $s \ges t$, and $t \mapsto B_t$ is weak
operator continuous, it is strong operator continuous and thus a local
$\al$-cocycle.
\end{proof}

\begin{remark}\label{subord E rem}
(a) The lack of injectivity for single maps noted after
Lemma~\ref{subord maps} does not arise for $E$-semigroups and
cocycles: if $P$ is a projection-valued, local $\al$-cocycle then
\[
P_t = P_t \al_t(P_0) = P_t \al_t (I) \ \text{ so } \ P_t \les \al_t
(I).
\]

(b) It follows from this result that the only $E_0$-semigroup
subordinate to an $E_0$-semigroup $\al$ is $\al$ itself.
\end{remark}

\begin{example}
Let $\Hil$ be a Hilbert space and $L = (L_t)_{t \ges 0}$ a strongly
continuous semigroup of isometries on $\Hil$. Put $\al_t (S) := L_t S
L^*_t$, then $\al$ is an $E$-semigroup. The only projection-valued,
local $\al$-cocycle is $(\al_t(I))_{t \ges 0}$, so that the only
subordinate $E$-semigroups of $\al$ is $\al$ itself.
\end{example}

\section{Cocycles of CCR flows}

The most understood class of $E$-semigroups are the CCR and CAR flows;
these comprise the type I examples, with the work of Powers et al.\ 
(\cite{poNY, app, 2-1}) designed to yield new examples of type II
semigroups. Here we revert to a consideration of the CCR flow,
bringing stochastic methods to bear.

Let $\init$ and $\noise$ be a pair of Hilbert spaces, called the
\emph{initial space} and \emph{noise dimension space} respectively.
For each measurable set $I \subset [0,\infty[$ let $\Fock_I$ denote
the symmetric/bosonic Fock space over $L^2(I;\noise)$, with
$\Fock_{\Rplus}$ abbreviated to $\Fock$, so that
\[
\Fock \cong \Fock_I \ot \Fock_{I^c}.
\]
This isomorphism is conveniently effected via the correspondence
\[
\e{f} \longleftrightarrow \e{f|_I} \ot \e{f|_{I^c}}
\]
defined in terms of the useful total set of exponential vectors:
\[
\Exp_S := \{\e{f}: f \in S\}, \text{ where } \e{f} :=
(1,f,(2!)^{-1/2} f \ot f, (3!)^{-1/2} f \ot f \ot f, \ldots),
\]
and where $S$ is any sufficiently large subset of $L^2(\Rplus;
\noise)$. For the purposes of this paper we shall write $\Exp$ for
$\Exp_S$ in the case that $S$ is the set of (right-continuous) step
functions in $L^2(\Rplus; \noise)$. The natural unitary isomorphism
$L^2(\Rplus; \noise) \cong L^2([t,\infty[; \noise)$ gives rise to the
identification
\[
\Fock \cong \e{0} \ot \Fock_{[t,\infty[} \subset \Fock_{[0,t[} \ot
\Fock_{[t,\infty[} \cong \Fock,
\]
and thence the right-shift map $\si^\noise_t$ on $B(\Fock)$ that has
image $I_{[0,t[} \ot B(\Fock_{[t,\infty[})$. This map is a normal
${}^*$-homomorphism. The CCR flow determined by $\init$ and $\noise$
is then
\[
\si^{\init,\noise} = \bigl(\si_t^{\init,\noise} := \id_{B(\init)} \vot
\, \si^\noise_t \bigr)_{t \ges 0}.
\]
[We now clarify tensor product notation: $\aot$ denotes the algebraic
tensor product, $\ot$ is used for the tensor product of Hilbert spaces
and the spatial tensor product, whereas $\vot$ denotes the ultraweak
tensor product.]

In general we will drop the dependence of the CCR flow on $\init$ and
$\noise$ in what follows, referring, for example, to $\si$-cocycles.

\begin{definition}
A $\si$-cocycle $X$ is:
\begin{alist}
\item
\emph{adapted} if $X_t \in B(\init \ot \Fock_{[0,t[}) \ot
I_{[t,\infty[}$ for all $t \ges 0$,
\item
\emph{Markov-regular} if the expectation semigroup $T_t := \mathbb{E}
[X_t]$ on $\init$ is norm-continuous, where the vacuum conditional
expectation $\mathbb{E}$ is defined by
\[
\ip{u}{\mathbb{E} [S]v} = \ip{u \ot \e{0}}{S v \ot \e{0}}, \quad u,v
\in \init.
\]
\end{alist}
\end{definition}

\begin{remark}
A $\si$-cocycle $X$ is local if $X_t \in I_{\init} \ot
B(\Fock_{[0,t[}) \ot I_{\Fock_{[t,\infty[}}$ for each $t$, which is a
strictly stronger condition than adaptedness whenever $\dim \init >
1$.
\end{remark}

The following combines Theorem~6.6 of~\cite{aunt} with Theorem~7.5
of~\cite{mother}. We define $\khat := \Comp \op \noise$, and $\Delta
:= I_\init \ot P$, where $P$ is the projection $\khat \To \{0\} \op
\noise$.

\begin{theorem}\label{cocqsde}
There is a one-to-one correspondence between the set of
Markov-regular, contraction, left, adapted $\si$-cocycles and
\[
C_0(\init, \noise) := \{F \in B(\init \ot \khat): F +F^* +F^* \Delta F
\les 0\},
\]
under which cocycles $X$ are associated to their \emph{stochastic
generator} $F$ through $X \longleftrightarrow X^F$, where $X^F$ is the
unique solution of the \emph{left Hudson-Parthasarathy quantum
stochastic differential equation}
\begin{equation}\label{qsde}
X_t = I + \int^t_0 \wh{X_s} (F \ot I_\Fock) \, d\Lambda_s,
\end{equation}
with $\wh{X_s}$ denoting the ampliation of $X_s \in B(\init \ot
\Fock)$ to all of $B(\init \ot \khat \ot \Fock)$.
\end{theorem}

The theorem above says that every such cocycle satisfies the
equation~\eqref{qsde}, and that all (contractive) solutions of this
equation are indeed left cocycles. In the case when $\init$ is finite
dimensional or the cocycle is local, the fact that we assumed that $t
\mapsto X_t$ is strongly continuous is enough to guarantee that the
cocycle is Markov-regular, but this condition is a nontrivial
requirement otherwise. Our basic reference for quantum stochastic
calculus is~\cite{qsc lects}, where the same notations are used.

Locality of $\si$-cocycles is also easily characterised.

\begin{lemma}\label{locality}
Let $F \in C_0(\init, \noise)$ and let $X = X^F$ be the associated
left $\si$-cocycle. Then $X$ is local if and only if $F \in I_\init
\ot B(\khat)$.
\end{lemma}

\begin{proof}
Since $\si_t \bigl( B(\init \ot \Fock) \bigr)' = I_\init \ot
B(\Fock_{[0,t[}) \ot I_{[t,\infty[}$, the result follows from
Corollary~6.5 of~\cite{aunt}.
\end{proof}

Hudson and Parthasarathy focused on determining those $F \in B(\init
\ot \khat)$ that give rise to unitary solutions of~\eqref{qsde}.
(Co)isometry and contractivity were considered later, with positive
contraction-valued and projection-valued adapted $\si$-cocycles
studied in~\cite{jtl} (although for the former see~\cite{ccr} for the
case when $\init = \Comp$). In Section~5.2 of~\cite{beast} the
morphisms of time ordered Fock modules are characterised;
Theorem~4.4.8 of the same paper gives a one-to-one correspondence
between endomorphisms of product systems of Hilbert modules and local
cocycles of an $E_0$-semigroup on a related $C^*$-algebra. These
results contain, as a special case, an alternative characterisation of
the local cocycles of the CCR flow.

The following is Proposition~3.2 of~\cite{jtl}. Throughout, when
writing elements of $B(\init \ot \khat)$ in a $2 \times 2$ block form,
we use the identification
\[
\init \ot \khat \cong \init \op (\init \ot \noise),
\]
so that
\[
F = \begin{bmatrix} A & B \\ C & D \end{bmatrix}
\]
for $A \in B(\init)$, $B \in B(\init \ot \noise; \init)$, $C \in
B(\init; \init \ot \noise)$ and $D \in B(\init \ot \noise)$.

\begin{proposition}\label{projcoc}
Let $F \in C_0(\init, \noise)$. The following are equivalent\tu{:}
\begin{rlist}
\item
$X = X^F$ is projection-valued\tu{;}
\item
$F \in \vN \, \vot B(\khat)$ for some commutative von Neumann algebra
$\vN$, and $F + F^* \Delta F = 0$\tu{;}
\item
$F = \begin{sbmatrix} -L^*L & L^* \\ L & P-I \end{sbmatrix} \in \vN \,
\vot B(\khat)$ for some commutative von Neumann algebra $\vN$, where
$P \in \vN \, \vot B(\noise)$ is a projection and $PL = 0$.
\end{rlist}
\end{proposition}

\begin{remark}
(a) The fact that $X$ is, in particular, self-adjoint means that it
must be a left and right cocycle, but not necessarily a local cocycle.
Indeed, using Lemma~\ref{locality}, an example is constructed by
taking $\init$ with $\dim \init > 1$, $\noise = \Comp$, $L = 0$, and
$P \in B(\init)$ any nontrivial projection.

(b) When $\init = \Comp$, we have $L \in B(\init; \init \ot \noise) =
B(\Comp; \noise) = \ket{\noise}$, the column operator space associated
to $\noise$, where $\ket{u} \in \ket{\noise}$ denotes the map $\lambda
\mapsto \lambda u$, and $\bra{u} := \ket{u}^*$. In this case
projection-valued, adapted $\si$-cocycles, which are now all local,
are indexed by pairs $(P,u)$, where $P \in B(\init)$ is a projection,
and $u \in \Ker P$ is a vector.
\end{remark}

The new result in this section is the following.

\begin{theorem}\label{compareprojcocs}
Let $X^F$ and $X^G$ be a pair of Markov-regular,
projection-valued, adapted $\si$-cocycles with generators
\begin{equation}\label{twogens}
F = \begin{bmatrix} -L^*L & L^* \\ L & P-I \end{bmatrix} \ \text{ and
} \ G = \begin{bmatrix} -M^*M & M^* \\ M & Q-I \end{bmatrix}.
\end{equation}

\noindent
\tu{(a)} The following are equivalent\tu{:}
\begin{rlist}
\item
$X^F_t \les X^G_t$ for all $t \ges 0$, equivalently $X^F_t X^G_t =
X^F_t$ for all $t \ges 0$.
\item
$G + G\Delta F = 0$.
\item
$P \les Q$ and $M = Q^\perp L$.
\end{rlist}

\noindent
\tu{(b)} Suppose that $F, G \in \vN \, \vot B(\khat)$ for some
commutative von Neumann algebra $\vN$, then the following are
equivalent\tu{:}
\begin{rlist}
\item
There is some $H \in \bigl( \vN \, \vot B(\khat) \bigr) \cap
C_0(\init, \noise)$ such that
\begin{equation}\label{equivgens}
X^F_t = (X^H_t)^* X^H_t \ \text{ and } \ X^G_t = X^H_t (X^H_t)^* \
\text{ for all } t \ges 0.
\end{equation}
\item
The projections $P$ and $Q$ are equivalent in $\vN \, \vot B(\noise)$,
that is, there is some $D \in \vN \, \vot B(\noise)$ such that $P =
D^*D$ and $Q = DD^*$.
\end{rlist}
In this case any $H$ for which~\eqref{equivgens} holds has the form $H
= \begin{sbmatrix} A & B \\ C & D-I \end{sbmatrix}$ with $D$ as
above and where
\begin{equation*}
C = M + E, \quad
B = L^* -E^*D, \ \text{ and } \ 
A = -\tfrac{1}{2} (E^*E +BB^* +C^*C) +iK
\end{equation*}
for some $E \in \vN \, \vot \ket{\noise}$ satisfying $Q^\perp E = 0$
and some $K = K^* \in \vN$.
\end{theorem}

\begin{remark}
The condition $E \in \vN \, \vot \ket{\noise}$ effectively means that
it can be written in block form as a column with entries taken from
$\vN$.
\end{remark}

\begin{proof}
(a) Let $\xi = \sum_i u_i \ot \e{f_i}, \zeta = \sum_j v_j \ot \e{g_j}
\in \init \aot \Exp$. Then the first and second fundamental formulae
of quantum stochastic calculus (\cite{qsc lects}, Theorems~3.13
and~3.15) give that
\[
\ip{X^F_t \xi}{X^G_t \zeta} - \ip{\xi}{X^F_t \zeta}
\]
is equal to
\begin{gather}
\begin{aligned} \int^t_0
\Bigl\{ &\ip{\wh{X^F_s} \xi(s)}{(G \ot I_\Fock) \wh{X^G_s} \zeta(s)} +
\ip{(F \ot I_\Fock) \wh{X^F_s} \xi(s)}{\wh{X^G_s} \zeta(s)} \\
+ &\ip{(F \ot I_\Fock) \wh{X^F_s} \xi(s)}{(\Delta G \ot I_\Fock)
\wh{X^G_s} \zeta(s)} -\ip{\xi(s)}{(F \ot I_\Fock) \wh{X^F_s} \zeta(s)}
\Bigr\}\, ds
\end{aligned} \label{projtest2} \\
= \int^t_0 \Bigl\{\ip{\xi(s)}{\wh{X^F_s} [(F +G +F \Delta G) \ot
I_\Fock] \wh{X^G_s} \zeta(s)} - \ip{\xi(s)}{(F \ot I_\Fock) \wh{X^F_s}
\zeta(s)} \Bigr\}\, ds \notag
\end{gather}
where $\xi(s) = \sum_i u_i \ot \wh{f_i (s)} \ot \e{f_i}$ and similarly
for $\zeta(s)$, with $\wh{d} := \begin{spmatrix} 1 \\ d \end{spmatrix}$
for any $d \in \noise$.
Thus if~(i) holds then the integral above is zero, and since the
integrand is continuous in a neighbourhood of $0$ we get
\[
\ip{\xi(0)}{[(F +G +F \Delta G) \ot
I_\Fock] \zeta(0)} - \ip{\xi(0)}{(F \ot I_\Fock) \zeta(0)} = 0
\]
from which~(ii) follows since $\xi(0)$ and $\zeta(0)$ range over a
total subset of $\init \ot \khat \ot \Fock$.

Let $\vN$ be a commutative von Neumann algebra as in
Proposition~\ref{projcoc}. To see that (ii) implies (i), note that
in~\eqref{projtest2} the term in square brackets equals $F \ot I_\Fock
\in (\vN \, \vot B(\khat)) \ot I_\Fock$, which commutes with
$\wh{X^F_s} \in (\vN \ot I_\khat) \vot B(\Fock)$. Consequently $X^F
X^G$ and $X^F$ are both solutions of the right version of the
Hudson-Parthasarathy QSDE~\eqref{qsde} for the same coefficient $F$,
and so uniqueness of solutions gives~(i).

Finally, (ii) is equivalent to the following four equations being
satisfied:
\[
M^*L = M^*M, \quad M^*P = 0, \quad M = Q^\perp L \ \text{ and } \ PQ =
P.
\]
That all four together are equivalent to just the last two is a
consequence of the fact that $QM = 0$, equivalently $Q^\perp M = M$,
and also $PL = 0$.

\bigskip
(b) This time, for any $H \in \bigl( \vN \, \vot B(\khat) \bigr) \cap
C_0(\init, \noise)$, it follows that $(X^H)^* X^H$ and $X^H (X^H)^*$
are both left and right cocycles, with stochastic generators $H +H^*
+H^* \Delta H$ and $H + H^* +H \Delta H^*$ respectively. (The argument
follows similar lines to part~(a); see also Lemma~3.1 of~\cite{jtl}.)
Uniqueness of generators implies that an $H$ exists such that~(i)
holds if and only if
\[
H +H^* +H^* \Delta H = F \ \text{ and } \ H + H^* +H \Delta H^* = G.
\]
Writing $H$ in block form as in the statement, this becomes
\begin{align}
\begin{bmatrix} A+A^*+C^*C & B+C^*D \\ B^*+D^*C & D^*D-I \end{bmatrix}
&= \begin{bmatrix} -L^*L & L^* \\ L & P-I \end{bmatrix}
\label{equivtest1} \\
\intertext{ and }
\begin{bmatrix} A+A^*+BB^* & C^*+BD^* \\ C+DB^* & DD^*-I \end{bmatrix}
&= \begin{bmatrix} -M^*M & M^* \\ M & Q-I \end{bmatrix}.
\label{equivtest2}
\end{align}
This shows immediately that (i) implies (ii). 

To see that (ii) implies (i), we have to satisfy a total of eight
equations when comparing components in~\eqref{equivtest1}
and~\eqref{equivtest2}. However, (ii) is merely the statement that it
is possible to find $D \in \vN \, \vot B(\noise)$ such that the
bottom-right components are equal in both equations. Comparing
bottom-left components we need
\begin{equation}\label{equivtest3}
B^*+D^*C = L \ \text{ and } \ C+DB^* = M.
\end{equation}
But note that $DD^* = Q$ and $D^*D = P$, so $D$ is a partial isometry
with initial projection $P$ and final projection $Q$. Since $PL = 0$,
it follows that $DL = 0$ as well; similarly $D^*M = 0$. Eliminating $B$
from~\eqref{equivtest3} yields a necessary condition on $C$:
\[
(I-DD^*)C = M -DL \ \text{ so } \ Q^\perp C = M.
\]
Since $Q^\perp M = M$, all solutions to this equation have the form $C
= M +E$ with $E \in \vN \, \vot \ket{\noise}$ satisfying $Q^\perp E =
0$. But now putting this back into the first equation
in~\eqref{equivtest3} we get
\[
B^* = L - D^*M -D^*E  \ \text{ so } \ B = L^* -E^*D.
\]
Now the bottom-left components of~\eqref{equivtest1} are equal by
construction, hence so are the top-right components. For the
off-diagonal components in~\eqref{equivtest2} we have
\[
C +DB^* = M +E +DL -DD^*E = M +DL +Q^\perp E = M
\]
as required.

Finally, to ensure that the top-left components in~\eqref{equivtest1}
and~\eqref{equivtest2} are equal, it is clear that twice the real part
of $A$ must equal both $-L^*L -C^*C$ and $-M^*M -BB^*$, with the
imaginary part unconstrained. That these two conditions can be
satisfied simultaneously follows since
\begin{align*}
-L^*L -C^*C &= -L^*L -(M+E)^*(M+E) \\
&= -L^*L -M^*M -E^*E -M^*E -E^*M \\
&= -L^*L -M^*M -E^*E,
\end{align*}
since $M^*E = (Q^\perp M)^* E = M^*Q^\perp E = 0$, and similarly
\begin{align*}
-M^*M -BB^* &= -M^*M -(L^*-E^*D)(L^*-E^*D)^* \\
&= -M^*M -L^*L - E^*QE +E^*DL  +(DL)^*E \\
&= -L^*L -M^*M -E^*E.
\end{align*}
Thus, if (ii) holds, that is if $P$ and $Q$ are equivalent in $\vN \,
\vot B(\noise)$, then it is possible to fill out the matrix for $H$ in
a way such that~\eqref{equivtest1} and~\eqref{equivtest2} both hold,
hence~(i) holds.  Moreover, the computations above reveal all possible
solutions to the problem.
\end{proof}

\begin{remark}
(a) Part~(a) above is in some sense more satisfactory that part~(b)
since we must have $F \in \vM \vot B(\khat)$ and $G \in \vN \, \vot
B(\khat)$ for a pair of commutative von Neumann algebras $\vM$ and
$\vN$, but no relation between these algebras is imposed. Indeed, take
$\noise = \Comp^2$, and $\init$ of dimension at least $2$ so that we
can pick noncommuting projections $p, q \in B(\init)$. Then let
\[
F = \begin{bmatrix} 0 & 0 \\ 0 & P-I \end{bmatrix}, \: \text{ and } \:
G = \begin{bmatrix} 0 & 0 \\ 0 & Q-I \end{bmatrix} \: \text{ where }
\: P = \begin{bmatrix} p & 0 \\ 0 & 0 \end{bmatrix} \: \text{ and } \:
Q = \begin{bmatrix} I_\init & 0 \\ 0 & q \end{bmatrix}.
\]
Since $G +G \Delta F = 0$, we have $X^F \les X^G$. However, any von
Neumann algebra $\vN_1$ that satisfies $F,G \in \vN_1 \vot B(\khat)$
must contain both $p$ and $q$. The assumption in~(b) about the
existence of the common commutative von Neumann algebra is to
facilitate pushing the coefficient $H$ past the cocycle $X^H$ in the
proof; when $\init = \Comp$, however, this is not actually a
restriction.

(b) In part (b), since $X^F$ and $X^G$ are projection-valued
cocycles, it follows that $X^H$ must be partial isometry-valued. A
necessary and sufficient condition on $H \in \vN \, \vot B(\khat)$ for
it to be the generator of a partial isometry-valued cocycle
(\cite{jtl}, Proposition~3.3) is that
\[
H +H^* +H^* \Delta H + H \Delta H + H \Delta H^* + H \Delta H^* \Delta
H = 0.
\]
The keen reader is invited to check directly that an $H$ satisfying
the structure relations in part~(b) does indeed satisfy this equation.
\end{remark}

\section{Relations on the subordinates of an $E$-semigroup}

For a given $E$-semigroup $\al$ on some $B(\Hil)$, let $\Sub(\al)$
denote the set of $E$-semigroups that are subordinate to $\al$. We now
turn our attention to possible natural relations that can be defined
on the set of all $E$-semigroups on $B(\Hil)$, or perhaps just on
$\Sub(\al)$ for some such semigroup $\al$.

The most obvious is clearly the relation $\les$ of subordination
itself. It is immediate from the definition that it is a partial order
on the set of all $E$-semigroups, hence on each subset $\Sub(\al)$.

\begin{lemma}\label{projles}
Let $\al$ be an $E$-semigroup and $\be, \ga \in \Sub(\al)$. Then $\ga
\les \be$ if and only if $\ga_t(I) \les \be_t(I)$ for all $t \ges 0$.
\end{lemma}

\begin{proof}
If $\ga \les \be$ then in particular $(\be_t-\ga_t) (I) \ges 0$, and
so $\ga_t(I) \les \be_t(I)$.

Conversely, suppose that we have $\ga_t(I) \les \be_t(I)$ for each
$t$. Now by Theorem~\ref{subord Esg}, the families $(\be_t(I))_{t \ges
0}$ and $(\ga_t(I))_{t \ges 0}$ are projection-valued, local
$\al$-cocycles, and also each $\be_t(I)-\ga_t(I)$ is a projection,
with
\[
\be_t(S) -\ga_t(S) = \bigl(\be_t(I) -\ga_t(I)\bigr) \al_t(S) =
\bigl(\be_t(I) -\ga_t(I)\bigr) \al_t(S)  \bigl(\be_t(I)
-\ga_t(I)\bigr)
\]
for all $S \in B(\Hil)$, so that $\be_t -\ga_t$ is CP as required.
\end{proof}

\begin{remark}
The temptation to adjust the hypotheses above and deal with all
$E$-semigroups rather than those from $\Sub(\al)$ should be
resisted. Whilst it is true that given \emph{any} two $E$-semigroups
$\be$ and $\ga$, if $\be \ges \ga$ then $\be_t(I) \ges \ga_t(I)$, the
converse is not true: if $\be$ and $\ga$ are $E_0$-semigroups then
obviously $\be_t(I) = \ga_t(I)$, but we need not have $\be =
\ga$.
\end{remark}

Let us apply this to the set $\Sub(\si)$ for the CCR flow $\si =
\si^{\init,\kil}$. By Theorem~\ref{subord Esg} any $E$-semigroup $\al$
subordinate to $\si$ is determined by a projection-valued, local
$\si$-cocycle $(X^\al_t = \al_t(I))_{t \ges 0}$, and from
Lemma~\ref{locality}, Theorem~\ref{cocqsde} and
Proposition~\ref{projcoc}, this in turn is specified uniquely by a
pair $(P_\al,u_\al)$ where $P_\al \in B(\noise)$ is a projection, and
$u_\al \in \Ker P_\al$, through
\begin{equation}\label{localprojgen}
X^\al = X^F \ \text{ for } \ F = I_\init \ot \begin{bmatrix}
-\norm{u_\al}^2 & \bra{u_\al} \\ \ket{u_\al} &
-P_\al^\perp \end{bmatrix}.
\end{equation}
So, from part~(a) of Theorem~\ref{compareprojcocs}, we get the
following.

\begin{proposition}
Suppose that $\al, \be \in \Sub(\si)$ with associated
projection-valued, local $\si$-cocycles $X^\al$ and $X^\be$, whose
stochastic generators have the form~\eqref{localprojgen}. Then $\al
\les \be$ if and only if
\begin{rlist}
\item
$P_\al \les P_\be$ and
\item
$u_\al = u_\be +u_{\al\be}$ for some $u_{\al\be} \in P_\be(\noise)
\cap P_\al^\perp(\noise)$.
\end{rlist}
\end{proposition}

\begin{example}\label{L2eg}
Let $\init = \Comp$ and $\noise = L^2[0,1]$. For each $r \in [0,1]$
let $P_r$ denote the projection $P_r f = f \ind_{[0,r]}$ on $\noise$.
Then define $\al^{(r)}$ to be the $E$-semigroup associated to the
projection-valued, local $\si$-cocycle with stochastic generator
\[
F_r = \begin{bmatrix} 0 & 0 \\ 0 & P_r -I \end{bmatrix}.
\]
It follows that for all $r \les s$ in $[0,1]$ we have $\al^{(r)} \les
\al^{(s)}$, that is, $\Sub(\si)$ contains an uncountable, linearly
ordered subset. This can only happen since we have taken an infinite
dimensional $\noise$, although the noise dimension space $\noise$ is
separable. In the case of a finite dimensional $\noise$ the maximum
number of distinct semigroups in a chain in $\Sub(\si)$ is $1 +\dim
\noise$.
\end{example}

In the theory of $E$-semigroups, a trick to overcome such features of
relations is to consider semigroups up to some form of equivalence.

\begin{definition}\label{cocconj}
$E$-semigroups $\al$ and $\be$ on $B(\Hil)$ are \emph{cocycle
conjugate} if there is a left $\al$-cocycle $U$ such that
\begin{rlist}
\item
$\be_t(S) = U_t \al_t(S) U^*_t$ for all $t \ges 0, S \in B(\Hil)$, and
\item
$\al_t(I) = U^*_t U_t$.
\end{rlist}
This is denoted $\al \sim \be$, or $\al \sim_U \be$ if we want to
highlight the particular cocycle.
\end{definition}

Most of the literature in this area deals solely with
$E_0$-semigroups, when a seemingly different definition of cocycle
conjugacy is found: $E_0$-semigroups $\al$ and $\be$ are called
cocycle conjugate if there is a unitary, left $\al$-cocycle $U$ such
that condition~(i) above holds. However, note that if $\al$ and $\be$
are $E_0$-semigroups above, then conditions~(i) and~(ii) together
imply that the $U_t$ are unitaries. Since we are forced to deal in
this paper with general $E$-semigroups (recall Remark~\ref{subord E
rem} (b)), we need the version of cocycle conjugacy given in
Definition~\ref{cocconj}, where the $U_t$ are partial isometries
courtesy of~(ii), but not necessarily unitaries.

\begin{proposition}
The relation $\sim$ is an equivalence relation on the set of all
$E$-semigroups on $B(\Hil)$.
\end{proposition}

\begin{proof}
We have $\al \sim_U \al$ for $U_t = \al_t(I)$. If $\al \sim_U \be$
then
\[
U^*_{s+t} = \al_s(U^*_t) U^*_s = \al_s(I) \al_s(U^*_t) U^*_s = U^*_s
U_s \al_s(U^*_t) U^*_s = U^*_s \be_s(U^*_t),
\]
so that $U^*$ is a left $\be$-cocycle, with $\be_t(I) = U_t \al_t(I)
U^*_t = (U_t U^*_t)^2 = U_t U^*_t$, and $U^*_t \be_t(S) U_t =
\al_t(S)$, so that $\be \sim_{U^*} \al$. Similar computations show
that if $\al \sim_U \be$ and $\be \sim_V \ga$ then $\al \sim_{VU} \ga$,
where $VU = (V_tU_t)_{t \ges 0}$ is a left $\al$-cocycle.
\end{proof}

The following relation was then discussed in~\cite{poNottm}.

\begin{definition}
Let $\Hil$ be a Hilbert space and let $\al, \be$ be $E$-semigroups on
$B(\Hil)$. Write $\al \cles \be$ if $\al \sim \al'$ for some other
$E$-semigroup $\al'$ for which $\al' \les \be$.
\end{definition}

\begin{proposition}\label{poOrder}
The relation $\cles$ is reflexive and transitive.
\end{proposition}

\begin{proof}
Reflexivity is immediate: $\al \sim \al$ and $\al \les \al$, so that
$\al \cles \al$. To establish transitivity assume that $\al \cles \be$
and $\be \cles \ga$. So there are $E$-semigroups $\al'$ and $\be'$ and
cocycles $U$ and $V$ such that
\[
\al \sim_U \al', \quad \al' \les \be, \quad \be \sim_V \be' \ \text{
and } \ \be' \les \ga.
\]
Set $V'_t = V_t \al'_t(I)$, then, using Theorem~\ref{subord Esg},
\[
V'_{s+t} = V_s \be_s (V_t) \al'_s \bigl(\al'_t(I)\bigr) = V_s \be_s
(V_t) \be_s \bigl(\al'_t(I)\bigr) \al'_s(I) = V'_s \al'_s (V'_t),
\]
so that $V'$ is a left $\al'$-cocycle. Moreover
\[
(V'_t)^* V'_t = \al'_t(I) V^*_t V_t \al'_t(I) = \al'_t(I) \be_t(I)
\al'_t(I) = \al'_t(I).
\]
Consequently, if we define
\[
\al''_t(S) = V'_t \al'_t(S) (V'_t)^*,
\]
then $\al''$ is an $E$-semigroup, with $\al' \sim_{V'} \al''$ by
construction. Also
\[
\be'_t(S) -\al''_t(S) = V_t \be_t(S) V_t^* -V'_t \al'_t(S) (V'_t)^* =
V_t (\be_t -\al'_t) (S) V_t^*,
\]
so that $\al'' \les \be'$. But then $\al \sim \al' \sim \al''$ and
$\al'' \les \be' \les \ga$, showing that $\al \cles \ga$.
\end{proof}

However, as noted in~\cite{poNottm}, the problem is proving that
$\cles$ is antisymmetric on the set of all $E$-semigroups considered
up to cocycle conjugacy. That is, if $\al \cles \be$ and $\be \cles
\al$ do we necessarily have $\al \sim \be$? New results of Liebscher
suggests that it is \emph{not} antisymmetric (\cite{nonsym}).

Inspired by the results of Section~2, we instead introduce an
alternative relation on the subset $\Sub(\al)$, rather than deal with
all $E$-semigroups simultaneously.

\begin{definition}
Let $\al$ be an $E$-semigroup on $B(\Hil)$, and let $\be, \ga, \de \in
\Sub(\al)$. We define $\be \sim^\al \ga$ if there is a local
$\al$-cocycle $U$ such that
\[
\be_t(I) = U^*_t U_t \ \text{ and } \ \ga_t(I) = U_t U^*_t \ \text{
for all } t \ges 0.
\]
We write $\be \sim^\al_U \ga$ to highlight the particular
$\al$-cocycle $U$. In addition we define $\be \cles^\al \de$ if there
is some $\be' \in \Sub(\al)$ such that $\be \sim^\al \be'$ with $\be'
\les \de$.
\end{definition}

These relations behave similarly to the versions without the
superscript $\al$.

\begin{proposition}
The relation $\sim^\al$ on $\Sub(\al)$ is an equivalence relation. The
relation $\cles^\al$ is reflexive and transitive.
\end{proposition}

\begin{proof}
If $\be \in \Sub(\al)$ then $U_t = \be_t(I)$ is a local $\al$-cocycle
for which $\be \sim^\al_U \be$, so that $\sim^\al$ is
reflexive. Symmetry is obvious, since if $U$ is a local $\al$-cocycle
then so is $U^*$. Finally, transitivity follows since if $U$ and $V$
are local $\al$-cocycles, then so is the pointwise product $VU =
(V_tU_t)_{t \ges 0}$, from which we quickly get that if $\be
\sim^\al_U \ga$ and $\ga \sim^\al_V \de$, then $\be \sim^\al_{VU}
\de$.

For the putative partial order $\cles^\al$, reflexivity follows as for
$\cles$. For transitivity, suppose that $\be, \ga, \de \in \Sub(\al)$
with $\be \cles^\al \ga$ and $\ga \cles^\al \de$, through the use of
local $\al$-cocycles $U$ and $V$ for which
\[
\be \sim^\al_U \be', \quad \be' \les \ga, \quad \ga \sim^\al_V \ga' \
\text{ and } \ \ga' \les \de.
\]
Set $V'_t = V_t \be'_t(I)$, a pointwise product of local
$\al$-cocycles, hence itself a local $\al$-cocycle. Moreover
\[
\Ran \be'_t(I) \subset \Ran \ga_t(I) = \text{initial space of } V_t,
\]
so that $V'$ is partial isometry-valued with initial projection
$\be'_t(I)$. Let $P_t = V'_t (V'_t)^*$, a projection-valued, local
$\al$-cocycle, and let $\be'' \in \Sub(\al)$ be the $E$-semigroup
determined by this $P$ through Theorem~\ref{subord Esg}. Then
$\be''_t(I) = P_t = V'_t (V'_t)^*$, hence $\be \sim^\al_U \be'
\sim^\al_{V'} \be''$. Moreover,
\[
\be''_t(I) = V_t \be'_t(I) V_t^* \les V_t V_t^* = \ga'_t(I)
\]
so that $\be'' \les \ga' \les \de$. Thus $\be \cles^\al \de$ as
required.
\end{proof}

For the CCR flow the missing piece of the puzzle is provided by
Theorem~\ref{compareprojcocs}.

\begin{proposition}
For the CCR flow $\si$, if $\al, \be \in \Sub(\si)$ with $\al
\cles^\si \be$ and $\be \cles^\si \al$ then $\al \sim^\si \be$.
\end{proposition}

\begin{proof}
Let $\al, \be \in \Sub(\si)$ with associated projection-valued, local
$\si$-cocycles $X^\al$ and $X^\be$, whose stochastic generators are
$F_\al$ and $F_\be$, where
\[
F_\al = I_\init \ot \begin{bmatrix} -\norm{u_\al}^2 & \bra{u_\al} \\
\ket{u_\al} & -P_\al^\perp \end{bmatrix}, \quad F_\be = I_\init \ot
\begin{bmatrix} -\norm{u_\be}^2 & \bra{u_\be} \\ \ket{u_\be} &
-P_\be^\perp \end{bmatrix},
\]
in the notation~\eqref{localprojgen}. By part~(b) of
Theorem~\ref{compareprojcocs} we have $\al \sim^\si \be$ if and only
if $P_\al \sim P_\be$, that is if and only if $P_\al$ and $P_\be$ are
equivalent projections.

Now by assumptions there are $\al', \be' \in \Sub(\si)$ such that $\al
\sim^\si \al'$ with $\al' \les \be$, and $\be \sim^\si \be'$ with
$\be' \les \al$. Maintaining the same notation for the generators of
$X^{\al'}$ and $X^{\be'}$, we have, now using part~(a) of
Theorem~\ref{compareprojcocs} as well as Lemma~\ref{projles}, that
\[
P_\al \sim P_{\al'}, \ \quad P_{\al'} \les P_\be, \quad P_\be \sim
P_{\be'} \ \text{ and } \ P_{\be'} \les P_\al.
\]
That is, $P_\al \cles P_\be$ and $P_\be \cles P_\al$, where $\cles$
here denotes subequivalence of projections in $B(\noise)$. But this
shows that $P_\al \sim P_\be$, and so $\al \sim^\si \be$ as required.
\end{proof}

\begin{example}
Let $\{\al^{(r)}:r \in [0,1]\} \subset \Sub(\si^{\init,\kil})$ be the
$E$-semigroups from Example~\ref{L2eg} where $\init = \Comp$, $\kil =
L^2[0,1]$. Now $P_r \sim I$ for each $r \in ]0,1]$ (since $P_r = T^*_r
T_r$ and $T_r T^*_r = I$ where $(T_rf)(s) = \sqrt{r} f(rs)$), and so it
follows that $\al^{(r)} \sim^\si \si$.
\end{example}

\par\bigskip\noindent
{\bf Acknowledgment.} Comments from Martin Lindsay, Michael Skeide and
the referee received during the writing of this paper resulted in
several useful improvements and clarifications.

\end{document}